\documentclass[a4paper, 11pt]{amsart}
  
\usepackage{amsmath}
\usepackage{amssymb}
\usepackage{amsthm}

\usepackage{color}
\usepackage{fix-cm}
\usepackage{graphicx}
\usepackage{placeins}
\usepackage{rotating}
\usepackage{soul}
\usepackage{subfigure}
\usepackage{verbatim}
\usepackage[foot]{amsaddr}

\usepackage{pifont, geometry, fleqn}
\usepackage[numbers]{natbib}

\renewcommand{\cite}{\citep}
\newtheorem{corollary}{Corollary}

\newtheorem{theorem}{Theorem}

\DeclareMathOperator{\supp}{\supp} 

\newcommand{\intone}{\int \limits_{0}^1}

\renewcommand{\geq}{\geqslant}

\renewcommand{\epsilon}{\varepsilon}

\newcommand{\rb}{\right]} 
\newcommand{\lb}{\left[}

\title[]{\bf A simple derivation of the integrals of products of Legendre polynomials with
  logarithmic weight}

\date{}

\begin{document}

\author[]{Sebastian Schmutzhard-H\"ofler}
\address{Faculty of Mathematics, University of Vienna,
  Oskar-Morgenstern-Platz 1, 1090 Vienna, Austria}
\email{sebastian.schmutzhard-hoefler@univie.ac.at}

\begin{abstract} 
\noindent We explore integrals of products of Legendre
polynomials with a logarithmic weight function. More precisely, for
Legendre polynomials $P_m$ and $P_n$ of orders $m$ and $n$,
respectively, we provide simple derivations of the integrals
\begin{align*}
  \intone  P_n(2x-1) P_m(2x-1)\log (x)dx.
\end{align*}
\end{abstract}
\maketitle
\section{Introduction}
\noindent In this short note we, compute the integrals of
products of Legendre polynomials with logarithmic
weight. Specifically, we are interested in
\begin{align*}
N_{n,m}: = \intone P_n(2x-1)P_m(2x-1) \log (x)dx
\end{align*}
for $m,n \in \mathbb{N}$.
The Legendre polynomials $P_0,P_1,P_2\ldots,$ are defined recursively
(\cite{Abr65}, \emph{22.7.10}). The initial polynomials are given by
\begin{align*}P_0(x)& = 1, \\
  P_1(x)& = x.
\end{align*}
For $n\geq 1$ the polynomials $P_n$ are defined by the
three-term-recurrence 
\begin{align} \label{eq_legrecurrence}
  (n+1)P_{n+1}(x)& = (2n+1)x P_n(x) - nP_{n-1}(x).
\end{align} 
In this work,  we derive for  $n \neq m$,
\begin{align}\label{eq_Nnneqm}
  N_{n,m} = \frac{(-1)^{n+m+1}}{|n-m|(n+m+1)},
\end{align}
and for $m = n$ 
\begin{align} \label{eq_Nneqm}
  \begin{split}
  (2n+1) N_{n,n} &= (2n-1) N_{n-1,n-1} - \frac{2}{(2n-1) 2n (2n+1)}  \\&=
    -1 -2 \sum \limits_{j = 1}^n \frac{1}{(2j-1)2j(2j+1)}.
  \end{split}
\end{align}
The solution for the special case $m = 0$ and  $n>0$ is obtained in
\cite{blue1979legendre}, and given by
\begin{align*}
\intone P_n(2x-1) \log(x)dx =  \frac{(-1)^{n+1}}{n(n+1)}.  
\end{align*}
This has been generalized to integrals of single Jacobi
polynomials with various weight functions, see for example
\cite{fettis1985further, gatteschi1980some, gautschi1979preceding, kalla1982integrals,
  srivastava1995some}. In \cite{shyam1984some}[(4.6)],integrals of products of
Legendre polynomials with logarithmic weights are represented by finite sums involving
the Gamma function and the Digamma function. The result \eqref{eq_Nnneqm}
can be found in \cite{Abr65}[8.14.8,8.14.10] as integrals of
products of Legendre functions of the first and the second kind. On
the other hand $N_{n,n}$ is given in \cite{Abr65}[8.14.9] via the
derivative of the Digamma function. Albeit being a special case and
easily derivable thereof, the simple representation of
\eqref{eq_Nneqm} seems to be new. Moreover, the elementary derivations of \eqref{eq_Nnneqm} and
\eqref{eq_Nneqm} seem, to our surprise and the best of our knowledge,
to be missing in the literature. 

\section{Main results}
\noindent We give rigorous derivations of the values of the integrals
  \begin{align*}
    N_{n,m} := \intone  P_n(2x-1) P_m(2x-1) \log(x) dx
  \end{align*}
  for arbitrary $n, m \in \mathbb{N}$.
\begin{theorem}\label{thm_ngtm}
 For $n,m \in \mathbb{N}$, $n > m$,
\begin{align}\label{eq_ngtm}
 N_{n,m}  =  \frac{(-1)^{n+m+1}}{(n-m)(n+m+1)}.
\end{align}
\end{theorem}
\begin{proof}[Proof of Theorem \ref{thm_ngtm}]
  We prove the result by induction.
  \newline
  \newline
  \underline{Induction base case}
  \newline
  The base case for $m = 0$ and $n>m$ is, as discussed above, proven
  in \cite{blue1979legendre}:
  \begin{align} N_{n,0}  =  \intone  P_n(2x-1) \log(x)dx =
  \frac{(-1)^{n+1}}{n(n+1)} \label{eq_Nn0}.
  \end{align}
  \newline
  For $m = 1$ and $n>m$, we observe that $P_1(2x-1) = 2x-1$. It
  follows from Equation
  \eqref{eq_legrecurrence} that for $n\geq 1$
  \begin{align}
    P_n(2x-1) (2x-1) = \frac{n+1}{2n+1} P_{n+1}(2x-1) +
    \frac{n}{2n+1}P_{n-1}(2x-1). \label{eq_legrecurrence2c}
  \end{align}
  Substituting \eqref{eq_legrecurrence2c} into the definition of $N_{n,1}$ gives
  \begin{align*}
    \begin{split}
    N_{n,1} &:= \intone P_n(2x-1) (2x-1)\log(x) dx  \\&= \intone  \lb
    \frac{n+1}{2n+1} P_{n+1}(2x-1) + \frac{n}{2n+1}P_{n-1}(2x-1)\rb \log(x)dx
    \\ & = \frac{n+1}{2n+1} \intone P_{n+1}(2x-1) \log(x) dx +
    \frac{n}{2n+1} \intone P_{n-1}(2x-1)\log(x)dx 
    \\& =  \frac{n+1}{2n+1} N_{n+1,0} + \frac{n}{2n+1} N_{n-1,0}.
    \end{split}
  \end{align*}
  Substituting the results for
  $N_{n+1,0}$ and $N_{n-1,0}$, see Equation \eqref{eq_Nn0}, leads to
  \begin{align*}
    \begin{split}
    N_{n,1} &=  \frac{n+1}{2n+1} \frac{(-1)^{n+2}}{(n+1)(n+2)} +
      \frac{n}{2n+1} \frac{(-1)^n}{(n-1)n}  \\
      & = \frac{(-1)^{n+2}}{2n+1}\lb \frac{1}{n+2} + \frac{1}{n-1}\rb =
      \frac{(-1)^{n+2}}{2n+1}\frac{2n+1}{(n+2)(n-1)} \\&=
      \frac{(-1)^{n+2}}{(n+2)(n-1)},
    \end{split}
  \end{align*}
   which is the base case for $m = 1$.
  \newline
  \newline
  \underline{Induction hypothesis}
  \newline
  We assume the result to be valid for any $\tilde{m} \in \mathbb{N}$
  with $\tilde{m}<m$. Specifically, we assume
  \begin{align}\label{eq_indhyp}\begin{split}
    N_{n,m-1} &= \frac{(-1)^{n+m}}{(n-m+1)(n+m)} \textnormal{ for }
    n>m-1 \textnormal{ and } \\
    N_{n,m-2} &= \frac{(-1)^{n+m-1}}{(n-m+2)(n+m-1)}  \textnormal{ for
    } n>m-2.
    \end{split}
  \end{align}
  \newline
  \newline
  \underline{Induction step}
  \newline
  It follows from Equation \eqref{eq_legrecurrence} that for $m > 1$
  \begin{align} \label{eq_legrecurrence1b}
    m P_m(2x - 1) = (2m-1) (2x-1)
    P_{m-1}(2x-1) -(m-1) P_{m-2}(2x-1).
  \end{align}
  Multiplying $N_{n,m}$ by $m$ and substituting
  \eqref{eq_legrecurrence1b} yields
  \begin{align*}
    mN_{n,m} &= \intone P_n(2x-1) m P_m(2x-1)\log(x)dx    \\& = \intone P_n(2x-1)  (2m-1) (2x-1)
    P_{m-1}(2x-1)\log (x)dx \\&- \intone P_n(2x-1)(m-1) P_{m-2}(2x-1)
    \log (x)dx,
  \end{align*}
  which results in
  \begin{align} \label{eq_mPm}
    \begin{split}
      &m N_{n,m} =\\
   &(2m-1) \intone 
      (2x-1) P_n(2x-1) P_{m-1}(2x-1)\log (x)dx - (m-1) N_{n,m-2},
    \end{split}
  \end{align}
  where we made use of the definition of $N_{n,m-2}$.
  It follows from Equation \eqref{eq_legrecurrence} that for $n\geq 1$
  \begin{align} (2n+1) (2x-1)  P_n(2x-1) = (n+1)P_{n+1}(2x-1) + n
    P_{n-1}(2x-1) \label{eq_legrecurrence2a}
  \end{align}
  Multiplying Equation \eqref{eq_mPm} by $(2n+1)$ and subsituting
  \eqref{eq_legrecurrence2a} leads to
  \begin{align*}
    \begin{split}
&    (2n+1) m N_{n,m} = \\& (2m-1) \intone  \lb
    (n+1)P_{n+1}(2x-1) + n P_{n-1}(2x-1)\rb P_{m-1}(2x-1)\log (x) dx \\ & -
    (2n+1)(m-1) N_{n,m-2}.
    \end{split}
  \end{align*}
  Making use of the definition of $N_{n+1,m-1}$ and $N_{n-1,m-1}$ gives us
  \begin{align}\label{eq_2np1mNnm}
    \begin{split}
&    (2n+1) m N_{n,m}    = \\
&    (2m-1)(n+1) N_{n+1,m-1} + (2m-1) n N_{n-1,m-1} - (2n+1)(m+1)
      N_{n,m-2}.
    \end{split}
  \end{align}
  The assumption that $n>m$ implies that
  \begin{align*}
    n+1&>m-1, \\
    n-1&>m-1 \textnormal{ and }\\
    n&>m-2,
  \end{align*}
  hence the induction hypotheses \eqref{eq_indhyp} are valid for $N_{n+1,m-1}$, $N_{n-1,m-1}$
  and $N_{n,m-2}$. Substituting the respective results into Equation \eqref{eq_2np1mNnm}, we obtain
  \begin{align} \label{eq_2np1mNnma}
    \begin{split}
    (2n+1) m N_{n,m} = &(2m-1)(n+1) \frac{(-1)^{n+m+1}}{(n-m+2)(n+m+1)}
       \\& +(2m-1)n \frac{(-1)^{n+m-1}}{(n-m)(n+m-1)}  \\&
       -(2n+1)(m-1)\frac{(-1)^{n+m-1}}{(n-m+2)(n+m-1)}.
    \end{split}
  \end{align}
  Bringing the right hand side of \eqref{eq_2np1mNnma} on a common denominator shows that
  \begin{align*}
    (2n+1) m N_{n,m} = (-1)^{n+m+1}\frac{ (2n+1) m (n-m+2)(n+m-1)}{(n-m+2)(n+m+1)(n-m)(n+m-1)}.
  \end{align*}
  Dividing by $(2n+1) m$ and simplifying the fraction proves the result \eqref{eq_ngtm}.
\end{proof}
\begin{corollary}
  For $n,m \in \mathbb{N}$, $n \neq m$,
  \begin{align*}
    N_{n,m} = \frac{(-1)^{n+m+1}}{|n-m|(n+m+1)}.
  \end{align*}
\end{corollary}
\begin{proof}
  The case of $n>m$ is covered by Theorem \ref{thm_ngtm}.
  In the case of $m>n$, we obtain the result of the Corollary by
  observing that $N_{n,m} = N_{m,n}$.
\end{proof}
\begin{theorem} \label{thm_neqm}
  \begin{align}
    N_{0,0} & = -1, \label{eq_N00} \\
    3 N_{1,1} & = -\frac{4}{3} \label{eq_N11},
  \end{align}
  and for  $n >1$,
  \begin{align*}
  (2n+1) N_{n,n} &= (2n-1) N_{n-1,n-1} - \frac{2}{(2n-1) 2n (2n+1)}  \\&=
  -1 -2 \sum \limits_{j = 1}^n \frac{1}{(2j-1)2j(2j+1)}.
  \end{align*}
\end{theorem}
\begin{proof}[Proof of Theorem \ref{thm_neqm}]
  Equations \eqref{eq_N00} and \eqref{eq_N11} can be verified by using
  $P_0(2x-1) = 1$ and $P_1(2x-1) = 2x-1$ and by explicitly computing the integrals.
  For arbitrary $n>1$, it follows from Equation \eqref{eq_legrecurrence}
  that for $n>1$
  \begin{align} \label{eq_legrecurrence1a}
    P_n(2x-1) = \frac{2n-1}{n} (2x-1) P_{n-1}(2x-1) - \frac{n-1}{n} P_{n-2}(2x-1).
  \end{align}
  Replacing one of the factors $P_n(2x-1)$ in the defintion of $N_{n,n}$ by
  \eqref{eq_legrecurrence1a}  and using
  the definition of $N_{n,n-2}$ yields
  \begin{align} \label{eq_Nnn}
    \begin{split}
    N_{n,n} &= \intone P_n(2x-1)P_n(2x-1)\log(x) dx \\&=  \intone P_n(2x-1) \frac{2n-1}{n} (2x-1) P_{n-1}(2x-1) \log(x)dx
    - \frac{n-1}{n} N_{n,n-2}.
    \end{split}
  \end{align}
  It follows from Equation \eqref{eq_legrecurrence} that for $n\geq 1$
  \begin{align} \label{eq_legrecurrence2b}
    (2n+1) (2x-1)  P_n(2x-1) = (n+1)P_{n+1}(2x-1) + n P_{n-1}(2x-1).
  \end{align}
  By multiplication of Equation \eqref{eq_Nnn} with (2n+1) and substituting \eqref{eq_legrecurrence2b}
  we obtain
  \begin{align*}
    & (2n+1)N_{n,n} = \\&(2n-1) N_{n-1,n-1} +\frac{2n-1}{n}
    (n+1)N_{n+1,n-1}  - (2n+1) \frac{n-1}{n} N_{n,n-2}
  \end{align*}
  where we made use of the definition of $N_{n+1,n-1}$ and
  $N_{n,n-2}$. Using Theorem \ref{thm_ngtm} results in
  \begin{align*}
  (2n+1)N_{n,n} &=
    (2n-1) N_{n-1,n-1} - \frac{(2n-1)(n+1)}{2n(2n+1)} +
    \frac{(2n+1)(n-1)}{2n (2n-1)} = \\
    &(2n-1) N_{n-1,n-1} +
    \frac{-(2n-1)^2(n+1)+(2n+1)^2(n-1)}{(2n-1)2n(2n+1)}
  \end{align*}
  which can be easily simplified to
  \begin{align*}
      (2n+1)N_{n,n} &=
(2n-1) N_{n-1,n-1} -
    \frac{2}{(2n-1)2n(2n+1)}.
  \end{align*}
  Finally, recursively substituting the last equation yields, together
  with the result for $N_{0,0}$, that
    \begin{align*}
  (2n+1) N_{n,n} = -1 -2 \sum \limits_{j = 1}^n \frac{1}{(2j-1)2j(2j+1)},
    \end{align*}
    which concludes the proof.

\end{proof}

\bibliography{intloglegleg.bib}
\bibliographystyle{plain}

\end{document}